\documentclass[11pt]{article}
\usepackage[a4paper]{geometry}
\usepackage{graphicx}
\usepackage[latin1]{inputenc}
\usepackage{latexsym,enumerate}
\usepackage{amsmath,amsthm,amsopn,amstext,amscd,amsfonts,amssymb}
\usepackage{color}
\usepackage{natbib}
\RequirePackage{hyperref}

\newcommand{\esp}[2]{\mathbb{E}_{#1}\left(#2\right)}
\newtheorem{thm}{Theorem}[subsection]

\newtheorem{lema}[thm]{Lemma}

\newtheorem{rem}[thm]{Remark}
\newtheorem{ex}[thm]{Example}
\allowdisplaybreaks[2]

\author{ Ana K. Fermin  (Modal'X, Universit\'e Paris Nanterre, France)\\
	and\\ Carenne Lude\~na ( Universidad Jorge Tadeo Lozano, \\ Dpto. de Ciencias B\'asicas y Modelado, Bogot\'a)}

\date{January 2018}

\title{Probability bounds for active learning  \\in the regression problem}

\begin{document}
	\maketitle	
	\label{duc1}
		
	
\begin{abstract}
	In this article we consider the problem of active learning in the regression setting. That is,  choosing an optimal sampling scheme for the regression problem simultaneously
	with that of model selection. We consider a batch type approach and an on--line approach adapting algorithms  developed
	for the classification problem. Our main tools are concentration--type inequalities which allow us to bound the supreme of
	the deviations of the sampling scheme corrected by an appropriate weight function. 
\end{abstract}
	

\section{Introduction}

Consider  the following regression model
\begin{equation}\label{eq:1}
y_i=x_{0}(t_{i})+\varepsilon_i, \, i=1,\ldots,n
\end{equation}
where the observation noise $\varepsilon_i $ are i.i.d.
realizations of a certain random variable $\varepsilon$. 

The problem we consider in this article is that of estimating the real-valued function $x_{0}$ based on $t_1,\ldots, t_n$ and a subsample of size $N<n$ of the observations
$y_{1},\ldots,y_{n}$ measured at a well chosen subsample of $t_1,\ldots,t_n$. This is relevant when,
for example, obtaining the values of $y_i$ for each sample point
$t_i$ is expensive or time consuming.

In this article we propose  a statistical regularization approach
for selecting a good  subsample of the data in this regression setting  by introducing a
weighted sampling scheme (importance weighting) and an appropriate
penalty function over the sampling choices.

We begin by establishing basic results  for a fixed model, and then the problem of model selection and choosing a good
sampling set simultaneously. This is what is known
as active learning. We will develop two approaches. The first, a
batch approach  \citep[see for example][]{Sugiyama--Rubens}, assumes the sampling set is
chosen all at once, based on the minimization of a certain penalized
loss function for the weighted sampling scheme.  The second, an iterative approach
\citep{Beygelzimer}, considers  a two step iterative method choosing
alternatively the best new point to be sampled and the best model
given the set of points.

The weighted sampling scheme requires each data point $t_i$ to be sampled with a certain probability $p(t_i)$ which is assumed to be inferiorly bounded by a certain constant $p_{min}$. This constant plays an important role because it controls the expected sample size $\esp{}{N}=\sum_{i=1}^n p(t_i)>np_{min}$. However, it also is inversely proportional to the obtained error terms in the batch procedure (see Theorems \ref{teoconsistencia1}, \ref{teoconsistencia2} and the results in Section  \ref{subsection3} ), so choosing $p_{min}$ too small will lead to poor bounds. Thus essentially, the batch procedure aims at selecting the best subset of data points (points with high probability) for the user chosen error bound. In the iterative procedure this problem is addressed by considering a sequence of sampling probabilities $\{p_j\}$ where at each step  $j$   $p_j(t_i)$ is chosen  to be as big as the greatest fluctuation for this data point over the  hypothesis model for this step.

Following the active learning literature for the regression problem based on ordinary least squares (OLS) and weighted least squares learning (WLS)  \cite[see, for example,][and the references therein]{Sugiyama--Rubens, Sugiyama,Sugiyama1}  in this article we deal mainly with a linear regression setting and a quadratic loss function. This will be done by
fixing a spanning family $\{\phi_{j}\}_{j=1}^m$   and considering the
best $L^2$ approximation $x_{m}$ of $x_{0}$ over this family.
However, our approach is  based on empirical error minimization techniques and can be readily extended to consider other models whenever bounds in probability are available for the error term. With this in mind we have included a more general formulation for bounded noise models in Section \ref{subsection3}.

Our results are based on concentration--type
inequalities. Although variance
minimization techniques for choosing appropriate sub--samples are a
well known tool, giving adequate bounds in probability
and allow dealing with more general non finite model spaces. This is also true for the iterative procedure, where our results generalize previous ones obtained only in the classification setting for finite model spaces.

The article is organized as follows. In Section  \ref{prelim} we formulate the
basic problem and study the batch approach for simultaneous sample and
model selection. In Section  \ref{subsection4}  we study the iterative approach to
sample selection and we discuss effective sample size
reduction. Section \ref{apen} is devoted to  proofs of the main
results and technical lemmas.

\section{Preliminaries}\label{prelim}

\subsection{Basic assumptions} \label{subsection0}

We assume that  the observations noise $\varepsilon_i$ in (\ref{eq:1})  are i.i.d. realizations  of a
random variable $\varepsilon$ satisfying the moment condition
\begin{description}
\item[MC] Assume the r.v. $\varepsilon$ satisfies ${\rm I\!E}
\varepsilon=0,$ ${\rm I\!E} (|\varepsilon|^r/\sigma^r) \le r ! /2$
for all $r>2$  and ${\rm I\!E} (\varepsilon^2)=\sigma^2$.
\end{description}
It is important to stress that the observations depend on a fixed design $t_1, \ldots, t_n$. For this,  we  need some notation concerning this design.

For any functions $u$ and $v$, we define the empirical norm of $u$ and its empirical scalar product with $u$ by
 $$\| u \|_{n}^2=\frac{1}{n} \sum_{i=1}^n (u(t_i))^2, \quad\text{and}\quad
<u,v>_n=\frac{1}{n} \sum_{i=1}^n u(t_i) v(t_i).$$

With the above notation,   given any positive
function $r$, we also introduce   the empirical $r$--scalar product
$<u,v>_{n,r}=\frac{1}{n} \sum_{i=1}^n r(t_i)u(t_i) v(t_i)$ and
$\|u\|_{n,r}$  the associated empirical norm.

We also require the empirical max-norm
$$\| u \|_{\infty}=\max_i |u(t_i)|.$$

With a slight abuse of notation, for any vectors $u$ and $v$, we denote in the same way the normalized norm and the normalized scalar product
with $u$
 $$\| u \|_{n}^2=\frac{1}{n} \sum_{i=1}^n (u_i)^2 ,\quad\text{and}\quad<u,v>_n=\frac{1}{n} \sum_{i=1}^n u_i v_i.$$
If $u$ is a function and $v$ a vector, we can thus use the following scalar product:
$$
<u,v>_n=\frac{1}{n} \sum_{i=1}^n u(t_i) v_i.
$$
which amounts to identifying $u$ and $(u(t_i))_{i=1}^{n}$.

Finally, we also require some notation concerning the spectral norm of a matrix.  

For a matrix $A$ we denote $A^t$ its transpose. With this notation, we define the spectral norm of $A$ by $$\|A\|=\max_{\|v\|=1} \|Av\|=\sqrt{\rho(A^tA)}$$ where  $\rho(A^tA)$ stands for the spectral radius of $A^tA$.


\subsection{Discretization scheme}\label{subsection1}
To start with we  will consider the approximation of function $x_{0}$ over a
finite--dimensional subspace $S_m$. This subspace will be assumed to
be linearly spanned  by the set  $\{\phi_{j}\}_{j\in \mathcal{I}_m}\subset
\{\phi_{j}\}_{j\ge 1} $, with $\mathcal{I}_m$ a certain index set. Moreover, we shall, in general, be interested only in the vector $(x_0(t_i))_{i=1}^{n}$ which we shall typically denote just  by $x_0$ stretching notation slightly.

Over the spanning family $\{\phi_j\}$, we will assume the following properties hold

\begin{description}
\item[AB] There exists an increasing sequence $c_{m}$ such that
$\|\phi_{j}\|_{\infty}\le c_{m}$ for $j\le m$.
\end{description}

\begin{description}\item[AQ] There exists a certain density $q$ and  a positive
constant $Q$ such that $q(t_i)\le Q, i=1,\ldots,n$ and
$$\int \phi_{l}(t)\phi_{k}(t)q(t)\, dt=\delta_{k,l},$$
where
$$\delta_{k,l}=\left\{%
\begin{array}{ll}
    1, & \hbox{ if\, $k=l$;} \\
    0, & \hbox{if not.} \\
\end{array}
\right.$$
\end{description}
\begin{sloppypar}

We will also require the following discrete approximation
assumption. Let $G_{m}=[\phi_{j}(t_{i})]_{{i,j}}$ be the associated empirical
$n\times m$ Gram matrix (design matrix). We assume that $G_{m}^{t}D_{q}G_{m}$ is invertible and moreover that
$\frac{1}{n}G_{m}^{t}D_{q}G_{m}\to I_{m}$, where $D_{q}$ is the
diagonal matrix with entries $q(t_{i})$, for $i=1,\ldots,n$ and  $I_m$  the identity matrix of size $m$.

More precisely, we will assume 
\end{sloppypar}
\begin{description}
\item[AS] There exist
positive constants $\alpha$ and $c$, such that  $$
\|I_m-\frac{1}{n}G_{m}^{t}D_{q}G_{m} \| \le c n^{-1-\alpha}.$$
\end{description}

 Given [AQ], assumption [AS]   is a numerical approximation condition which is satisfied under certain regularity assumptions over $q$ and $\{\phi_j\}$. To illustrate this condition we include the next examples.

 \begin{ex}
 Real trigonometric polynomials: let  $t\in [0,1]$, $\phi_{0}(t)=1$, and for $k \in \mathbb{N}$ set

  $$
  \begin{array}{rl}
       \phi_{2k}(t) &= \sqrt{2} \cos(2\pi kt),  \quad k \geq 1 \\
     \phi_{2k+1}(t) &= \sqrt{2} \sin(2\pi kt), \quad k \geq 0 \\
  \end{array}
  .$$ 
 We assume $q(t)=\mathbf{1}_{[0,1]}(t)$,  and $S_m$ is the linear space spanned by the functions $\phi_k(t), k\leq m$. Then $c_m=\sqrt{2}$ and condition [AS] is  satisfied.
 \end{ex}
 \begin{ex}
 Haar Wavelets: let
  $\phi(t)=\mathbf{1}_{[0,1]}(t)$,  $\psi(t)=\phi(2t)-\phi(2t-1)$  \citep[see for
   example][]{LNS129}, with $q(t)=\mathbf{1}_{[0,1]}(t)$.  Define
 $$\begin{array}{rcl}
 \phi_{j,k}(t)& = & 2^{j/2}\phi(2^jt-k) \, , \; t \in [0,1]
 \, , \; j \geq 0 \, \textrm{and } k \in \mathbb{Z} \; ;\\
 \psi_{j,k}(t)& = & 2^{j/2}\psi(2^jt-k) \, , \; t \in [0,1] \,
 , \; j \geq 0 \, \textrm{and } k \in \mathbb{Z} \; .
 \end{array}$$
 For all $m \geq 0$, $S_m$ denotes the linear space spanned by the functions $(\phi_{m,k}, k \in \mathbb{Z})$. In this case $c_m\le2^{m/2}$ and condition [AS] is satisfied for the discrete sample $t_i=i/2^m$, $i=0,\cdots, 2^{m-1}$.
 \end{ex}

We will denote by $\hat{x}_{m} \in S_{m}$ the function that
minimizes the weighted norm $\|x-y\|^2_{n,q}$ over $S_{m}$ evaluated at points $t_1,\ldots,t_n$. This is,
\begin{equation}
\hat{x}_{m}=\mbox{arg}\min_{x\in S_{m}}\frac{1}{n}\sum_{i=1}^n q(t_i)
(y_i-x(t_i))^2=R_{m}y,\nonumber
\end{equation} with
$R_{m}=G_{m}(G_{m}^{t}D_{q}G_{m})^{-1}G_{m}^{t}D_{q}$  the
orthogonal projector over $S_m$ in the $q$--empirical  norm
$\|\cdot\|_{n,q}$.

Let $x_{m}:=R_{m}x_{0}$   be the projection  of $x_{0}$ over $S_{m}$
in the q--empirical norm $\|\cdot\|_{n,q}$, evaluated at points $t_1,\ldots,t_n$. Our goal is to choose a
good subsample of the data collection such that the estimator of the
unobservable vector $x_0$ in the finite--dimensional subspace
$S_m$, based on this subsample, attains near optimal error bounds. For
this we must introduce the notion of subsampling scheme and
importance weighted approaches  \cite[see][]{Beygelzimer,Sugiyama--Rubens}, which we discuss below.

\subsection{Sampling scheme and importance weighting}\label{subsection1.2}

In order to sample the data set we will introduce a sampling probability $p(t)$ and a sequence of
Bernoulli($p(t_{i})$) random variables $w_{i},\, i=1,\ldots,n$
independent  of $\varepsilon_{i}$ with $p(t_{i})>p_{\min}$. Let
$D_{w,q,p}$  be the diagonal matrix with entries
$q(t_{i})w_{i}/p(t_{i}) $. So that $\esp{}{D_{w,q,p}}=D_{q}$.
Sometimes it will be more convenient to rewrite $w_i=\mathbf{1}_{u_i<p(t_i)}$
for $\{u_i\}_i$ an i.i.d. sample of uniform random variables,
independent of $\{\varepsilon_i\}_i$ in order to stress the
dependence on $p$ of the random variables $w_i$.

The next step is to construct an estimator for  $x_m = R_m x_0$,
based on the observation vector $y$ and the sampling scheme $p$.
For this, we consider a
modified version of the estimator $\hat{x}_m$.

Consider a uniform random sample $u_1,\ldots,u_n$ and let $w_i=w_i(p)=\mathbf{1}_{u_i<p(t_i)}$ for a given $p$. For the given realization of $u_1,\ldots, u_n$, $D_{w,q,p}$ will be  strictly positive for those $w_i=1$. Moreover, as follows from the singular value decomposition, the matrix $(G_{m}^{t}D_{w,q,p}G_{m})$ is invertible as long as at least one $w_i\ne 0$.  

Set $R_{m,p}=G_{m}(G_{m}^{t}D_{w,q,p}G_{m})^{-1}G_{m}^{t}D_{w,q,p}$.  Then $R_{m,p}$ is
the orthogonal projector over $S_m$ in the $wq/p$--empirical  norm
 $\|\cdot\|_{n,wq/p}$ and it is well defined if  at least one $w_i\ne 0$. If all $w_i=0$ the projection is defined to be $0$.

 As the approximation of $x_m$, we then consider (for a fixed $m$, $p$ and $(u_1,\ldots,u_n)$ ) the random quantity
\begin{align} \nonumber
\hat{x}_{m,p} &=\mbox{arg}\min_{x\in S_{m}}
\|x-y\|^2_{n,\frac{qw}{p}}\\ \nonumber &=\mbox{arg}\min_{x\in
S_{m}}\frac{1}{n}\sum_{i=1}^n \frac{w_i}{p(t_i)}q(t_i)(y_i - x(t_i))^2
.\nonumber
\end{align}

\begin{equation}\label{eq:RIWE}
\hat{x}_{m,p}=R_{m,p}y,
\end{equation}
  Note that this
estimator depends on  $y_i$ only if $w_i=1$. However,  as stated above, this depends on $p(t_i)$ for the given probability $p$.

\subsection{Choosing a good sampling scheme}\label{subsection1.3}
To begin with, given $n$, we will assume  that 
$S_m$ is fixed with dimension $|\mathcal{I}_m|=d_m$ and $d_m=o(n)$ (in general $m$ is assumed constant). Remark that the bias $\|x_{0}-x_{m}\|^2_{n,q}$ is independent of $p$ so for our purposes it is only necessary to study  the
approximation error $\|x_{m}-\hat{x}_{m,p}\|^2_{n,q}$ which does depend on how the sampling probability $p$ is chosen.

Let $\mathcal{P}:=\{p_{k}, k\ge 1 \}$
be a numerable collection of   $[0,1]$ valued functions over $\{t_1,\ldots,t_n\}$. Set $p_{k,min}=\min_i p_{k}(t_{i})$. We will assume  that $\min_{k}p_{k,min}>p_{\min}$. The way the candidate probabilities are ordered is not a major issue, although in practice it is sometimes convenient to incorporate prior knowledge (certain sample points are known to be needed in the sample for example) letting favorite candidates appear first in the order. To get the idea of  what a sampling scheme may be, consider the following toy example:

  \begin{ex}
  Let  $\Pi=\{0.1, 0.4, 0.6, 0.9 \}$ and  set
 $$\mathcal{P} = \{ p,  p(t_i)=\pi_j \in \Pi, i=1,\ldots,n \}$$
 which is a set of $|\Pi|^{n}$ functions. In this example, any given $p$ will tend to favor the appearance of points $t_i$ with  $p(t_i)=0.9$ and disfavor the appearance of those $t_i$ with $p(t_i)=0.1$.
\end{ex}
A good sampling scheme $p$, based on the data, should be the minimizer over $\mathcal{P}$ of the non
observable quantity $ \| x_m-\hat{x}_{m,p}\|^2_{n,q}.$ In order to find a reasonable observable equivalent  we start by writing,
\begin{eqnarray}\label{eq:descerror}
&&[\hat{x}_{m,p}-x_m] \nonumber\\
&&=R_{m,p}[x_{0}-x_{m}]+R_{m,p}\varepsilon \nonumber \\&&
=\esp{}{R_{m,p}}[x_{0}-x_{m}]+(R_{m,p}-\esp{}{R_{m,p}})[x_{0}-x_{m}]+ R_{m,p}\varepsilon.
\end{eqnarray}

For the first equality we have used that
\begin{eqnarray*} R_{m,p}x_{m}&=&G_{m}(G_{m}^{t}D_{w,q,p}G_{m})^{-1}G_{m}^{t}D_{w,q,p}G_{m}(G_{m}^{t}D_{q}G_{m})^{-1}G_{m}^{t}D_{q}x_0\\&&
=G_{m}(G_{m}^{t}D_{q}G_{m})^{-1}G_{m}^{t}D_{q}x_0=x_m.
\end{eqnarray*}
Consider first the deterministic term $\esp{}{R_{m,p}}[x_{0}-x_{m}]$ in (\ref{eq:descerror}). We have the next lemma which is proved in the Appendix
\begin{lema}\label{lema3--1}
Under condition [AS]  if $m=o(n)$, then
$$\|\esp{}{R_{m,p}}[x_0-x_m]\|_{n,q}=O(\frac{n^{-1-\alpha}\|x_0-x_m\|_{n,q}}{ p_{min}}).$$
\end{lema}

From Lemma \ref{lema3--1}, we can derive  that the deterministic term is small with respect to the other terms. Thus, it is sufficient for a good sampling scheme to take into account the second and third terms in  (\ref{eq:descerror}).  We propose to use  an upper bound with high probability of those two last terms as in a penalized estimation scheme and to base our choice on this bound.

Define, 

\begin{equation}\label{pen1-1}\widetilde{B}_1(m,p_{k},\delta)= \|x_0-x_m\|^2_{n,q} (\widetilde{\beta}_{m,k}(1+\widetilde{\beta}_{m,k}^{1/2}))^2\end{equation}
with
\begin{equation}\label{deltam}\widetilde{\beta}_{m,k}=\frac{c_m(\sqrt{17}+1)}{2}\sqrt{\frac{d_m Q}{n p_{k,\min}}}\sqrt{2\log(2^{7/4}d_mk(k+1)/\delta)}.\end{equation}
The second square root appearing in the definition of $\widetilde{\beta}_{m,k}$ is included in order to give uniform bounds over the numerable collection $\mathcal{P}$. 

In the following,  the expression  $\mbox{tr}(A)$ stands for the trace of the matrix $A$. 
Set $T_{m,p_{k}}=\mbox{tr}((R_{m,p_{k}}D_q^{1/2})^t R_{m,p_{k}}D_q^{1/2})$ and define
\begin{eqnarray} \widetilde{B}_{2}(m,p_{k},\delta)&=& \sigma^{2}r (1+\theta_{k})\frac{T_{m,p_{k}}+Q}{n}
+ \sigma^{2} Q\frac{log^2(2/\delta)}{ d
	n}, \label{pen2}\end{eqnarray} with $r>1$ and  $d=d(r)<1$ a positive constant
that depends on $r$.  The sequence  $\theta_k\ge 0$  is such that
the Kraft condition $\sum_{k}e^{-\sqrt{dr \theta_{k}(d_m+1) }} < 1$
holds.

It is thus reasonable,
to consider the best $p$ as the minimizer
\begin{equation}\label{bestsampling}
\hat{p}= \underset{p_k \in \mathcal{P}}{\rm{argmin }} \widetilde{B}(m,p_k,\delta,\gamma,n),
\end{equation} where, for a given $0<\gamma<1$,
$$\widetilde{B}(m,p_k,\delta,\gamma,n)=\{(1+\gamma)\widetilde{B}_1(m,p_k,\delta) + (1+1/\gamma) \widetilde{B}_2(m,p_k,\delta)\}.$$

The different roles of $\widetilde{B}_1$ and $\widetilde{B}_2$ appear in the following lemmas:
\begin{sloppypar}
	\begin{lema}\label{lema03}
		Assume that the conditions [AB], [AS], and [AQ] are satisfied and that there is
		a constant $p_{min}
		>0$ such that  for all $i=1,\ldots,n$, $p (t_i)
		> p_{k,\min}>p_{\min}$.  Assume $\widetilde{B}_1$ to be selected according to (\ref{pen1-1}).
		Then for all $\delta>0$ we have
		\begin{eqnarray*}
			&&P\bigg[\sup_{\mathcal{P}}\{\|(R_{m,p}-\esp{}{R_{m,p}})[x_0-x_m]\|^2_{n,q} - \widetilde{B}_1(m,p,\delta)\}>0 \bigg]\le
			\delta/2
		\end{eqnarray*}
	\end{lema}
\end{sloppypar}

\begin{lema}\label{lema04} Assume the observation noise in equation
	\eqref{eq:1} is an i.i.d. collection of random variables satisfying
	the moment condition [MC]. Assume that the condition [AQ] is satisfied and assume  that there is a constant
	$p_{min}
	>0$ such that $p (t_i)
	> p_{min}$ for all $i=1,\ldots,n$. Assume $\widetilde{B}_2$ to be selected according to (\ref{pen2}) with $r>1$, $d=d(r)$
	and $\theta_k \geq 0$, such that  the following Kraft inequality $\sum_{k}e^{-\sqrt{d r \theta_{k}(m+1) }} < 1$ holds.  Then,
	\begin{eqnarray*}
		&& P(\sup_{\mathcal{P}}\{\|R_{m,p}\varepsilon\|^2_{n,q}- \widetilde{B}_2(m,p,\delta)\}>0)<\delta/2.
	\end{eqnarray*}
\end{lema}

Those two lemmas together with Lemma \ref{lema3--1}  assure that the proposed estimation procedure, based on the minimization of $\widetilde{B}$, is  consistent establishing non asymptotic rates in probability.

 We may now state the main result of this section, namely, non-asymptotic consistency rates in probability of the proposed estimation procedure. The proof follows from Lemmas~\ref{lema03} and~\ref{lema04} and is given in the Appendix along with the proof of the  lemmas.

\begin{thm}\label{teoconsistencia1} Assume that the conditions [AB], [AS]
and [AQ] are satisfied. Assume $\hat{p}$ to be selected according to (\ref{bestsampling}). Then the
following inequality holds with probability greater than $1-\delta$
\begin{eqnarray*}
	\|x_m-\hat{x}_{m,\hat{p}}\|^2_{n,q}
	&\le& \inf_{p \in \mathcal{P}} 6 \bigg( \|\esp{}{R_{m,p}}(x_m-x_0)\|^2_{n,q} + \widetilde{B}(m,p,\delta,\gamma,n)\bigg).
\end{eqnarray*}
\end{thm}

\noindent Certain remarks are in order.

\begin{rem}
	In the minimization scheme given above it is not necessary to know the term
	$\|x_0-x_m\|^2_{n,q}$ in $\widetilde{B}_1$  as this term is constant with regard to the sampling scheme $p$. Including this term in the definition of $\widetilde{B}_1$,  however, is important because it leads to optimal bounds in the sense that it balances $p_{min}$ with the  mean  variation, over the sample points, of the best possible solution $x_m$ over the hypothesis model set $S_m$. This idea shall be pursued in depth in Section  \ref{subsection4}.
	
	Moreover, minimizing $\widetilde{B}_1$ essentially just requires selecting $k$ such that $p_{k,\min}$ is largest  and doesn't intervene at all if $p_{k,\min}=p_{min}$ for all $k$. Minimization based on $p_k(t_i)$ for all sample points is given by   the  trace $T_{m,p_{k}}$ which depends on  the initial random sample $u$ independent of $\{(t_i,y_i), i=1,\ldots,n\}$. A reasonable strategy in practice, although we do not have theoretical results for it, is to consider several realizations of $u$ and select sample points which appear more often in the selected sampling scheme $\hat{p}$.
\end{rem}

\begin{rem}
	Albeit the appearance of  weight terms which depend on $k$ both in the definition of
	$\widetilde{B}_1$ and $\widetilde{B}_2$,
	actually the ordering of  $\mathcal{P}$ does not play a major role.
	The weights are given in order to assure convergence over the
	numerable collection $\mathcal{P}$. Thus in the definition of
	$\widetilde{\beta}_{m,k}$ any sequence of weights $\theta'_k$
	(instead of $[k(k+1)]^{-1}$) assuring that the series $\sum_k
	\theta'_k<\infty$ is valid. Of course, in practice $\mathcal{P}$ is finite. Hence for $M=|\mathcal{P}|$ a more reasonable bound is just to consider uniform weights $\theta'_k=1/M$ instead.
\end{rem}
\begin{rem}
	Setting
	$H_{m,p_{k}}:=(G_{m}^{t}D_{w,q,p_{k}}G_{m})^{-1}G_{m}^{t}D_{w,q,p_{k}}$ we may write $T_{m,p_{k}}=\mbox{tr}(G_m^tD_qG_m
	H_{m,p_{k}}H_{m,p_{k}}^t) $ in the definition of $\widetilde{B}_2$. Thus our convergence rates are as in  Lemma 1, \cite{Sugiyama}. Our approach however provides non-asymptotic bounds in probability as opposed to asymptotic bounds for the quadratic estimation error.
\end{rem}
\begin{rem}
	As mentioned at the beginning of this section, the expected "best" sample size given $u$  is
	$\hat{N}=\sum_i \hat{p}(t_i)$, where $u$ is the initial random sample independent of $\{(t_i,y_i), i=1,\ldots,n\}$. Of course, a uniform inferior bound for this expected sample size is $\esp{}{\hat{N}}>np_{min}$, so that the expected size is inversely proportional to the user chosen estimation error. In practice, considering several realizations of the initial random sample provides an empirical estimator of the non conditional "best" expected sample size.
\end{rem}

\subsection{Model selection and active learning}\label{subsection2}
\begin{sloppypar}
Given a model and $n$ observations $(t_1,y_1), \ldots,(t_n,y_n)$
 we know how to estimate the best sampling scheme
$\hat{p}$ and to obtain the estimator $\hat{x}_{m,\hat{p}}$. The
problem is that the model $m$ might not be a good one. Instead of
just looking at {\em fixed} $m$ we would like to consider
simultaneous model selection as in \cite{Sugiyama--Rubens}. For this
we shall pursue a more global approach based on loss functions.
\end{sloppypar}

 We start by introducing some notation. Set $l(u,v)=(u-v)^2$ the squared loss and let $L_n(x,y,p)=\frac{1}{n}\sum_{i=1}^n  q(t_i) \frac{w_i}{p(t_i)}l(x(t_i),y_i)$
 be the empirical loss function for the quadratic difference with the given sampling distribution.
  Set $L(x):=\esp{}{L_n(x,y,p)}$ with the expectation taken over all the random variables involved. Let $L_{n}(x,p):=\esp{\varepsilon}{L_n(x,y,p)}$ where $\esp{\varepsilon}{}$ stands for the conditional expectation given the initial random sample $u$, that is the expectation with respect to the random noise $\varepsilon$. It is not hard to see that
  $$L(x)=\frac{1}{n}\sum_{i=1}^n q(t_i) \esp{}{l(x(t_i),y_i)},$$ and
  $$L_{n}(x,p)=\frac{1}{n}\sum_{i=1}^n q(t_i) \frac{w_i}{p(t_i)}\esp{}{l(x(t_i),y_i)}.$$
 Recall that $\hat{x}_{m,p}=R_{m,p}y$ is the minimizer of $L_n(x,y,p)$ over each $S_m$ for given $p$ and  that $x_m=R_{m}x_0$ is the minimizer of $L(x)$ over $S_m$. Our problem is then to find the best approximation of the target
$x_0$ over the  function space $S_0:=\bigcup_{m\in \mathcal{I}} S_m$. In the
notation of Section \ref{subsection1} we assume for each $m$ that
$S_m$ is a bounded subset of the linearly spanned space of the
collection $\{\phi_j\}_{j\in I_m}$ with $|\mathcal{I}_m|=d_m$.

Unlike the fixed $m$ setting, model selection requires controlling not only the variance term  $\|x_m-\hat{x}_{m,p}\|_{n,q}$ but also the unobservable bias term $\|x_0-x_m\|^2_{n,q}$ for each possible model $S_m$.
 If all samples were available this would be possible just by looking at $L_n(x,y,p)$ for all $S_m$ and $p$, but in the active learning setting  labels are expensive.

Set $e_m:=\|x_0-x_m \|_\infty$. In what follows we will
assume that there exists a positive constant $C$ such that  $
\sup_{m} e_m\le C$. Remark this implies $\sup_{m}\|x_0-x_m\|_{n,q}\le QC$, with $Q$ defined in [AQ].

As above $p_{k}\in \mathcal{P}$ stands  for the set of candidate sampling probabilities and $p_{k,\min}=\min_i(p_{k}(t_i))$.

Define
\begin{equation}\label{pen0}
pen_{0}(m,p_k,\delta)= \frac{QC^2}{p_{k,\min}} \sqrt{\frac{1}{2 n }\ln(\frac{6d_m(d_m+1)}{\delta})},
\end{equation}
\begin{equation}\label{pen1new}
pen_1(m,p_k,\delta)= QC \beta^2_{m,k}(1+\beta_{m,k}^{1/2})^2 ,
\end{equation}
with
$$\beta_{m,k}=\frac{c_m(\sqrt{17}+1)}{2}\sqrt{\frac{d_m Q}{n p_{k,\min}}}\sqrt{2\log(\frac{3* 2^{7/4}d_m^2(d_m+1)k(k+1)}{\delta})},$$
and finally setting $T_{p_k,m}=\mbox{tr}((R_{m,p_k}D_q^{1/2})^t R_{m,p_k}D_q^{1/2})$, define
\begin{equation}\label{pen2new}
pen_2(m,p_k,\delta)=
\sigma^{2}\left\{r(1+\theta_{m,k})\frac{T_{p_k,m}+
 Q}{n}+\frac{Q\ln^2(6/\delta)}{d n}\right\}
\end{equation}
where $\theta_{m,k}\ge 0$  is a sequence such that
$\sum_{m,k}e^{-\sqrt{dr \theta_{m,k}(d_m + 1) }}<1$ holds.
 We remark that the
change from $\delta$ to $\delta/(d_m(d_m+1))$ in $pen_0$ and $pen_1$  is required in
order to account for the supremum over the collection
of possible model spaces $S_m$.

Also, we remark that introducing  simultaneous model and sample
selection  results in the inclusion of term $pen_0\sim C^2/
p_{k,\min}*\sqrt{1/n}$ which includes an $L_\infty$ type bound
instead of an $L_2$ type norm which may yield non optimal bounds.
 Dealing more efficiently with  this term would require knowing the (unobservable) bias term $\|x_0-x_m\|_{n,q}$.
A reasonable strategy is selecting  $p_{k,min}=p_{k,min}(m)\ge
\|x_0-x_m\|_{n,q}$ whenever this information is available.

In practice, $p_{k,min}$ can be estimated for each model $m$ using a
previously estimated empirical error over a subsample if this is
possible. However this yields a conservative choice of the bound.
 One way to  avoid this inconvenience is to consider iterative procedures, which update on the unobservable bias term. This course of action shall be pursued in Section \ref{subsection4}.

 With these definitions, for a given $0<\gamma<1$ set
 \begin{eqnarray*}&&pen(m,p,\delta,\gamma,n)\\
 &&=[2p_0(m,p,\delta)+(\frac{1}{p_{\min} }+\frac{1}{\gamma }) pen_1(m,p,\delta)\\
 &&+(\frac{1}{p_{\min}^2 }(\frac{2}{\gamma}+1)+\frac{1}{\gamma})pen_2(m,p,\delta)+2((c+1)\frac{n^{-(1+\alpha)}QC}{p_{min}})^2].
 \end{eqnarray*}
and define
 $$L_{n,1}(x,y,p)=L_n(x,y,p)+pen(m,p,\delta,\gamma,n).$$

 The appropriate choice of an optimal sampling scheme simultaneously with that of model selection is a difficult problem.
We would like to choose simultaneously $m$ and $p$, based on the data in such a way that optimal rates are maintained.  We propose for this a penalized version of $\hat{x}_{m,\hat{p}}$, defined as follows.

 We  start by choosing, for each $m$, the best sampling scheme
\begin{equation}\label{bestsampling1}
\hat{p}(m)=\mbox{arg}\min_{p}{pen}(m,p,\delta,\gamma,n),
\end{equation}
computable before observing the output values $\{y_i\}_{i=1}^n$, and then calculate the estimator
$\hat{x}_{m,\hat{p}(m)}=R_{m,\hat{p}(m)}y$ which was defined in \eqref{eq:RIWE}.

Finally, choose the best model as
 \begin{equation}\label{bestmodel}
 \hat{m}=\mbox{arg}\min_{m}L_{n,1}(y,\hat{x}_{m,\hat{p}(m)},\hat{p}(m)).
 \end{equation}

The penalized estimator is then $\hat{x}_{\hat{m}}:=\hat{x}_{\hat{m},\hat{p}(\hat{m})}$.
It is important to remark that for each model $m$, $\hat{p}(m)$ is independent of $y$ and hence of the random observation error structure.
The following result assures the consistency of the proposed estimation procedure, although the obtained rates are not optimal as observed at the beginning of this section. The proof is given in the Appendix.
\begin{thm}\label{teoconsistencia2}
With probability greater than $1-\delta$, we have
 \begin{eqnarray*}
 &&L(\hat{x}_{\hat{m}})\le \frac{1+\gamma}{1-4\gamma}[  L(x_m)\\
 &&+\min_{m,k}(2 p_0(m,p_k,\delta)+\frac{1 }{p_{\min}}pen_{1}(m,p_k,\delta)\\
 &&+\frac{1}{p_{\min}^2}(1+2/\gamma)pen_2(m,p_k,\delta))]\\
 &&\le\frac{1+\gamma}{1-4\gamma}\min_m [L(x_m)+\min_k pen(m,p_k,\delta,\gamma,n)]
 \end{eqnarray*}
 \end{thm}

\begin{rem}
 In practice,   a reasonable alternative to the proposed minimization procedure is estimating the overall error by cross--validation or leave one out techniques
  and then choose   $m$ minimizing the error for successive essays of probability $\hat{p}$.  Recall that
  in the original procedure of Section \ref{subsection2}, labels are not required to obtain $\hat{p}$ for a fixed model.
  Cross--validation or empirical error minimization techniques do however  require a stock of ``extra" labels, which might not be
  affordable in the active learning setting.  Empirical error minimization is specially useful for   applications where  what is required is a subset of very informative sample points, as for example when deciding what points  get extra labels   (new laboratory runs, for example) given a first set of complete labels is available.  Applications suggest that $\hat{p}$ obtained with this  methodology
    (or a threshold version of $\hat{p}$ which eliminates points with sampling probability
    $\hat{p}_i\le \eta$ a certain small constant) is very accurate in finding ``good" or informative subsets, over which model selection may be performed.
\end{rem}


\subsection{Error bounds for the general bounded case}\label{subsection3}
\begin{sloppypar}
The above procedure can be extended to other frameworks, defined by
minimization over $S_0=\bigcup S_m$ of a given loss function
$$L_n(x,y,p)=\frac{1}{n}\sum_{i=1}^n q(t_i)
\frac{w_i}{p(t_i)}l(x(t_i),y_i)$$ with expectation
$L(x)=\esp{}{L_n(x,y,p)}=\frac{1}{n}\sum_{i=1}^n q(t_i)\esp{}
{l(x(t_i),y_i)}$. Set, as above,
$L_n(x,p)=\esp{\varepsilon}{L_n(x,y,p)}=\frac{1}{n}\sum_{i=1}^n
q(t_i)\frac{w_i}{p(t_i)}\esp{} {l(x(t_i),y_i)}$. We will denote
$l(x)=\esp{} {l(x,y)}$. In order to repeat the proof of Section
\ref{subsection2} it is necessary to control both the fluctuations
of $D_n:=L_n(x,y,p)-L_n(x',y,p)-[L_n(x,p)-L_n(x',p)]$
   and   $D_n':=L_n(x,p)-L_n(x',p)-[L(x)-L(x')]$. Dealing with the first term $D_n$  requires   bounding
    $$\Delta(m,p):=\sup_{x\in S_m}\frac{1}{n}\sum_{i=1}^n q(t_i)\frac{w_i}{p(t_i)}[l(x(t_i),y_i)-l(x(t_i))].$$
\end{sloppypar}
 Assuming $l(x,y)$ is uniformly bounded by a constant $B^2$ (which is not the case for the example presented in Section \ref{subsection2}) standard arguments allow improving the results given in the last subsection  \cite[see for example][for a very thorough discussion]{essaimsurvey}, although it is never possible in the batch approach setting to eliminate the term $p_{min}$ from the denominator of the error. More precisely,  combining the bounded differences inequality and bounds for Rademacher sums lead to the following bound in probability
 $$P(\Delta(m,p_k)-[t_{m,k}+2\esp{}{\Omega_n(m,k)}]>0)\le \frac{\delta}{k(k+1) m(m+1)},$$ with
 $$t_{m,k}=\sqrt{\frac{4 B^4\log(2 m(m+1)k(k+1)/\delta)}{n p_{\min}^2}}$$ and
 $$\Omega_n(m,k)=\esp{\xi}{\sup_{x\in S_m}\frac{1}{n}\sum_{i=1}^n q(t_i)\frac{w_i}{p(t_i)}\xi_i l(x(t_i),y_i)}$$ where $\xi_i$ is a sequence of independent (Rademacher) random variables, $P(\xi=-1)=P(\xi=1)=1/2$ and independent of $y_i$. A uniform bound over the candidate model spaces and probabilities then yields,
 \begin{equation}\label{cota51} P(\sup_{m,k}[\Delta(m,p_k)-(t_{m,k}+2\esp{}{\Omega_n(m,k)})]>0)<\delta
 \end{equation}
 Dealing with the second term $D_n'$ on the other hand requires bounding
 $$\Delta'(m,k):=\sup_{x\in S_m}\frac{1}{n}\sum_{i=1}^n q(t_i)(\frac{w_i}{p(t_i)}-1)l(x(t_i)). $$
 Again if $l(x)$ is bounded, combining the bounded differences inequality and bounds for Rademacher sums lead to
 $$P(\Delta'(m,p_k)-[t_{m,k}+2\esp{}{\Omega_n'(m,k)}]>0)\le \frac{\delta}{k(k+1) m(m+1)},$$ with
 $$t_{m,k}=\sqrt{\frac{4 B^4\log(2 m(m+1)k(k+1)/\delta)}{n p_{\min}^2}}$$ and
 $$\Omega_n'(m,k)=\esp{\xi}{\sup_{x\in S_m}\frac{1}{n}\sum_{i=1}^n q(t_i)\frac{w_i}{p(t_i)}\xi_i l(x(t_i))}$$ which yields the uniform bound,
 \begin{equation}\label{cota52} P(\sup_{m,k}[\Delta'(m,P_k)-(t_{m,k}+2\esp{}{\Omega_n'(m,k)})]>0)<\delta
 \end{equation}
 Actually, the bounds in Section \ref{subsection2} follow from bounding the Rademacher sums   by $\Omega_n(m,k)\le C Tr(\Omega_{n,P_k})$,
 for a certain constant $C$ and using Lemma \ref{lema1} in the Appendix. In turn, this lemma  follows from a functional exponential inequality proved in \cite{Bousquet}. Hence, it would seem that equation (\ref{cota52}) does not add any interesting information to what has already been discussed. However, the more general setting is important
 because (in the bounded case) allows  passing from Rademacher sums to V--C dimensions  \cite[see again][for a general discussion]{essaimsurvey}  which then makes it possible to  consider more general solution spaces than a numerable union of target model spaces $S_m$. Bounds in this case are given by
 \begin{equation}\label{cota61} P(\sup_{x\in S_0,p}[L_n(x,y,p)-L(x)] >2[\sqrt{\frac{4 B^4\log(2  /\delta)}{n p_{\min}^2}}+2 \sqrt{\frac{2\log(n)V}{np_{\min}}}])<\delta
 \end{equation}
  where $V$ is the V--C dimension of the class of functions $S_0$.

 \section{Iterative procedure: updating the sampling probabilities}\label{subsection4}
 A major drawback of the batch procedure   is the appearance of  $p_{\min}$   in the denominator of error bounds, since typically $p_{\min}$ must be small in order for the estimation procedure to be effective. Indeed, since the expected number of effective samples is given by $\esp{}{N}:=\esp{}{\sum_i p(t_i)}$, small values of $p(t_i)$ are   required in order to gain in sample efficiency.

Proofs in Sections \ref{subsection2} and  \ref{subsection3} depend heavily on bounding expressions such as

  $$\frac{1}{n}\sum_{i=1}^n q(t_i) \frac{w_i}{p(t_i)}\varepsilon_i (x-x')(t_i)$$ or
  $$\frac{1}{n}\sum_{i=1}^n q(t_i) (\frac{w_i}{p(t_i)}-1) (x-x')^2(t_i).$$
  Thus, it seems like a reasonable alternative to consider
    iterative procedures for which at time $j$,
  $p_j(t_i)\sim \max_{x,x'\in S_j}|x(t_i)-x'(t_i)|$ with $S_j$ the current hypothesis space. In what follows we develop this strategy, adapting the results of \cite{Beygelzimer} from the classification  to the regression problem. Although we continue to work in the setting of model selection over bounded subsets of linearly spanned spaces, results can be readily extended to other frameworks such as additive models or kernel models. Once again, we will require certain additional restrictions associated to the  uniform approximation of $x_0$ over the target model space.

  More precisely. We start with an initial model set $S(=S_{m_0})$ and set $x^*$ to be the overall   minimizer of the loss function $L(x)$ over $S$. Assume additionally

   \begin{description}

       \item[AU]  $\sup_{x\in S}\max_{t \in \{t_1,\ldots,t_n\}}|x_0(t )-x(t )|\le B$
   \end{description}

   Let $L_n(x)=L_n(x,y,p)$ and $L(x)$ be as in Section \ref{subsection2}. For the iterative procedure introduce the notation
   $$L_{j}(x):=\frac{1}{n_j}\sum_{i=1}^{nj}q(t_{j_i})\frac{w_{i}}{p(t_{j_i})}(x(t_{j_i})-y_{j_i})^2, \quad j=0,\ldots,n$$ with  $n_j=n_0+j$ for $j=0,\ldots,n-n_0$. 
   
   In the setting of Section \ref{prelim}  for each $0\le j\le n$, $S_j$ will be the linear space spanned  by the collection $\{\phi_\ell\}_{\ell\in \mathcal{I}_j}$ with $|\mathcal{I}_j|=d_j$, $d_j=o(n)$.

   In order to bound the fluctuations of the initial step in the iterative procedure we consider the quantities defined in equations (\ref{pen1-1}) and (\ref{pen2})
   for $r=\gamma=2$. That is,  \begin{eqnarray*}\Delta_0&=&2\sigma^2 Q\left\{\frac{2(d_0+1)}{n_0}+\frac{\log^2(2/\delta)}{n_0}\right\}\\
   &&+2(\widetilde{\beta}_{m_0}(1+\widetilde{\beta}_{m_0}))^2B^2 .
   \end{eqnarray*}
   with $$\widetilde{\beta}_{m_0}=\frac{c_{m_0}(\sqrt{17}+1)}{2}\sqrt{\frac{d_{0} Q}{n_0 p_{\min}}}\sqrt{2\log(2^{7/4}m_0/\delta)}.$$

   As discussed in Section \ref{subsection1.3}, $\Delta_0$ requires some initial guess of $\|x_0-x_{m_0}\|^2_{n,q}$.
   Since this is not available, we consider the upper bound $B^2$. Of course this will possibly slow down the initial convergence as $\Delta_0$ might be too big,
   but will not affect the overall algorithm. Also remark we do not consider the weighting sequence $\theta_k$ of equation (\ref{pen2}) because the sampling probability is assumed fixed.

   Next set $B_j=\sup_{x, x'\in S_{j-1}}\max_{t \in \{t_1,\ldots,t_n\}}|x(t )-x'(t )|$ and define

  \begin{eqnarray*}&&\Delta_j=\sqrt{\sigma^2 Q[(\frac{ 2 (d_j+
 1)}{n_j})+\frac{\log^2(4n_j(n_j +1)/\delta)}{n_j}]}\\
 &&+\sqrt{\log(4n_j(n_j +1)/\delta)\frac{16B_j^2(2B_j\wedge 1)^2Q^2}{n_j}}+4\sqrt{4\frac{(d_j+1)\log  n}{n_j}}.
 \end{eqnarray*}

    The iterative procedure is stated as follows:
  \begin{enumerate}
  \item For $j=0$:
   \begin{itemize}
   \item Choose (randomly) an initial sample of size $n_0$, $M_0=\{t_{k_{1}},\ldots, t_{k_{n_0}}\}$.
   \item Let $\hat{x}_0$ be the chosen solution by minimization of $L_{0}(x)$ (or possibly a weighted version of this loss function). \item  Set $S_0\subset\{x\in S: L_0(x)<L_0(\hat{x}_0)+\Delta_0\}$
       \end{itemize}
      \item \label{pasos}At step $j$:
      \begin{itemize}
      \item Select (randomly) a sample candidate point $t_{j}$, $t_{j}\not\in M_{j-1}$. 
      
      Set $M_j=M_{j-1}\cup\{t_{j}\}$
      \item Set $p(t_j)=(\max_{x,x'\in S_{j-1}}|x(t_j)-x'(t_j)|\wedge 1)$ and generate $w_j\sim Ber(p(t_j))$. 
      
      If $w_j=0$, set $j=j+1$ and go to (\ref{pasos}) to choose a new sample candidate. 
      
       If $w_j=1$ sample $y_{j}$    and continue.
      \item  Let $\hat{x}_j=\mbox{arg}\min_{x\in S_{j-1}}L_j(x)+ \Delta_{j-1}(x)$
     \item  Set $S_j\subset\{x\in S_{j-1}:L_j(x)< L_j(\hat{x}_j)+\Delta_j\}$
     \item Set $j=j+1$ and go to (\ref{pasos}) to choose a new sample candidate.
       \end{itemize}
  \end{enumerate}

  Remark, that such as it is stated, the procedure can continue only up until time $n$ (when there are no more points to sample).
  If the process is stopped at time $T<n$ the term $\log(n(n+1))$ can be replaced by $\log(T(T+1))$.
We have the following result, which generalizes Theorem 2 in \cite{Beygelzimer} to the regression case. The proof is given in the Appendix.

  \begin{thm}\label{teoiterativo}Let $x^*=\mbox{arg}\min_{x\in S} L(x)$. Set $\delta>0$. Then, with probability at least $1-\delta$ for any $j\le n$
  \begin{itemize}
  \item $|L(x)-L(x')|\le 2\Delta_{j-1}$, for all $x,x'\in S_{j}$
  \item $L(\hat{x}_j)\le   [L(x^*)+2\Delta_{j-1}]$
  \end{itemize}
  \end{thm}

\begin{rem}
An important issue is related to the initial choice of $m_0$ and $n_0$. As the overall precision of the algorithm is determined by $L(x^*)$, it is important to select a sufficiently complex initial model collection.
However, if $d_{m_0}>>n_0$ then $\Delta_0$ can be big and $p_j\sim 1$ for the first samples, which leads to a more inefficient sampling scheme.
\end{rem}

  \subsection{Effective sample size}\label{subsection5}

  For any sampling scheme the expected number of effective samples is, as already mentioned,  $\esp{}{\sum_i p(t_i)} $.
  Whenever the sampling policy is fixed, this sum is not random and effective reduction of the sample size will depend on how small sampling probabilities are.
   However, this will increase the error bounds as a consequence of the factor $1/p_{\min}$. The iterative
   procedure allows a closer control of both aspects and under suitable conditions will be of order $\sum_j\sqrt{ L(x^*)+ \Delta_j}$, as we will prove below establishing appropriate bounds over the random sequence $p(t_j)$. Recall from the definition of the iterative procedure we have $p_j(t_i)\sim \max_{x,x'\in S_j}|x(t_i)-x'(t_i)|$, whence the expected number of effective samples is of the order of $\sum_j \max_{x,x'\in S_j}|x(t_i)-x'(t_i)|$.
 It is then necessary to control $\sup_{x,x'\in S_{j-1}}|x(t_i)-x'(t_i)|$ in terms of the (quadratic) empirical loss function  $L_j$. For this we must introduce some notation and results relating the supremum and $L_2$ norms \cite{birgemassart}.

   Let $S\subset L_2\cap L_\infty$ be a linear subspace of dimension $d$,  with basis $\Phi:=\{\phi_j, j\in m_S\}$, $|m_S|=d$. Set $\eta(S):=\frac{1}{\sqrt{d}}\sup_{x\in S,x\ne 0}\frac{\|x\|_\infty}{\|x\|_2}$, ${r}_\Phi:=\frac{1}{\sqrt{d}} \sup_{\beta,  \beta\ne 0}\frac{\sum_{\phi_j\in \Phi}\beta_j\phi_j}{|\beta|_\infty}$ and $\overline{r}:=\inf_\Lambda r_\Lambda$, where $\Lambda$ stands for any orthonormal basis of $S$. We require  the following result in  \cite{birgemassart}:
   \begin{lema}\label{lemamassartbirge} \cite[Lemma 1][]{birgemassart} Let $S$ be a $d$ dimensional linear subspace of $ L_2\cap L_\infty$, with basis $\Phi$,
   and set $\eta(S):=\frac{1}{\sqrt{d}}\sup_{t\in S }\|t\|_\infty/\|t\|_2$. Then
   \begin{enumerate}
   \item  $\eta(S)=\|\sum_{\phi_j \in \Phi}\phi_j^2\|_\infty^1/2$
   \item $\eta(S)\le \overline{r}\le \eta(S) \sqrt{d}$
   \end{enumerate}
   \end{lema}

    \begin{ex}
    Some examples of $\overline{r}$ for typical linear settings include \cite[ pp 337--338]{birgemassart}:
    \begin{enumerate}
   \item Trigonometric expansions: $\overline{r}\le \sqrt{2d}$.
   \item Polynomials: $\overline{r}\le d$.
   \item Localized basis:
   \begin{itemize}
   \item $\{\phi_j=\sqrt{d} 1_{[(j-1)/d,j/d]}\}_{1\le j\le d}$: $\overline{r}\le 1$
   \item Piecewise polynomials on $[0,1]$ of degree m: $\overline{r}\le 2m+1$
   \item Orthonormal wavelet systems: $\overline{r}\le C$, for a certain constant $C$ depending on the form of the basis.
   \end{itemize}

    \end{enumerate}
    \end{ex}

    We have the following result
    \begin{lema}\label{simplesize}
    Let $\hat{x}_j$ be the sequence of iterative approximations to $x^*$ and $p_j(t)$ be the sampling probabilities in each step of the iteration, $j=1,\ldots,T$. Then, the effective number of samples, that is, the expectation of the required samples  $N_e=\esp{}{\sum_{j=1}^T p_j(t_j)}$ is bounded by
    $$N_e\le 2\sqrt{2}\overline{r}( \sqrt{L(x*)}\sum_{j=1}^T \sqrt{d_j}+\sum_{j=1}^T \sqrt{d_j\Delta_j}).$$
    \end{lema}

 The proof is given in the Appendix. 

\section{Appendix}\label{apen}

\begin{proof} Proof of Lemma \ref{lema03}

 We will achieve the proof by bounding
 $$\|(R_{m,p}-\esp{}{R_{m,p}})[x_0-x_m]\|^2_{n,q}\le \|R_{m,p}-\esp{}{R_{m,p}}\|^2 \|x_0-x_m\|^2_{n,q}.$$
 
 For this we shall consider a double application of  a straightforward generalization
 of Theorem 7.3 in \cite{rauhut}, whose proof is given below.
\begin{sloppypar}
\begin{lema}\label{lemarauhut} Let $A\in \mathbb{R}^{n\times m}$ be some matrix whose  rows, $a(l)\in \mathbb{R}^{ m}$, l=1,\ldots n,  satisfy  $\|a(l)\|_2\le K\sqrt{m}$
 for some constant $K \geq 1$. Consider the matrix $A^tA=\sum_{l=1}^n a(l) a(l)^t$ and let $\Lambda_{A}=\frac{1}{n}\|\esp{}{A^tA}\|$. Set $\tau=(\sqrt{17}+1)/4$. We have the following bounds:
 \begin{itemize}
\item For any $r\ge 1$, define $E_r:=\esp{}{\|\frac{1}{n
\Lambda_{A}}(A^tA-\esp{}{A^tA})\|^r}$  and let
$$\Sigma_{r,m,n}=\left(\frac{2K}{\sqrt{n\Lambda_{A}}}\sqrt{\frac{m}{n}}\right)^r
2^{3/4} m r^{r/2} e^{-r/2}.$$ Then for any $r \geq 2$,
$$E_r^{1/r} \leq   \tau \Sigma_{r,m,n}.$$
\item Let $\delta<1/2$, then the following bound in probability holds true for $u\ge \sqrt{2}$
$$P(\|\frac{A^tA-\esp{}{A^tA}}{n \Lambda_{A}}\| > \frac{2\tau K}{\sqrt{\Lambda_{A}}}\sqrt{\frac{m}{n}}u)\le m 2^{3/4}e^{-u^2/2},$$
or equivalently with probability at least $1-\delta$
$$\|\frac{A^tA-\esp{}{A^tA}}{n \Lambda_{A}}\| \leq \frac{2\tau K}{\sqrt{\Lambda_{A}}}\sqrt{\frac{m}{n}}\sqrt{2log(2^{3/4}m/\delta)}.$$
\end{itemize}
\end{lema}
\end{sloppypar}
\begin{sloppypar}
With this lemma we continue the proof of Lemma \ref{lema03}.

 Recall that $R_{m,p}=\frac{1}{n}G_{m}(\frac{1}{n}G_{m}^{t}D_{w,q,p}G_{m})^{-1}G_{m}^{t}D_{w,q,p}$.
On the other hand, observe that since $A_{m,p}:=1/nG_m^t D_{pqw}G_m$ is a positive definite matrix its inverse exists
and moreover we may write $A_{m,p}^{-1/2}$ using the standard spectral notation. Also since $A_{m,p}$ is symmetric we have $A_{m,p}^{-1/2}=(A_{m,p}^{-1/2})^t$.
\end{sloppypar}
Consider the matrix $\tilde{A}_{m,p}= D_{pqw}^{1/2}G_m$. Then we have its rows $\tilde{a}_{m,p}(l)$ satisfy
$$\|\tilde{a}_{m,p}(l)\|=\sqrt{\frac{w_lq(t_l)}{p(t_l)}\sum_{j=1}^m (\phi_j(t_l))^2}\le c_m\sqrt{m\frac{Q}{p_{min}}}$$ and
$\Lambda_{\tilde{A}_{m,p}}=\|\esp{}{1/n
\tilde{A}^t_{m,p}\tilde{A}_{m,p}}\|\le 1+c(n)^{-\alpha-1}$
under assumption [AS].

In what follows set $k=k(p)$ and let $\delta'_k=\delta/(2k(k+1))$.
 A first application of Lemma
\ref{lemarauhut} then yields
$$\|A_{m,p}-\esp{}{A_{m,p}}\| \leq 2\tau c_m \sqrt{1+c(n)^{-\alpha-1}} \sqrt{\frac{m Q}{n p_{min}}}\sqrt{2\log(2^{3/4}m/\delta'_k)}$$ with probability greater  $1-\delta'_k$.
Here, the choice of $\delta'_k$ is required in order to account for
the supremum over the collection of possible sampling schemes.

It then follows using a classical Neumann series expansion that with
probability greater than $1-\delta'_k$,
\begin{equation}\label{cotainversa}
\|A_{m,p}^{-1}\|\le
\frac{1}{1-(\sqrt{1+cn^{-\alpha-1}}\widetilde{\beta}_{m,k} +cn^{-\alpha-1})},
\end{equation}
where $$\widetilde{\beta}_{m,k}=2\tau c_m\sqrt{\frac{m Q}{n
p_{\min}}}\sqrt{2\log(2^{3/4}m/\delta'_k)}.$$

Now, consider the matrix  $\tilde{E}_{m,p}= A_{m,p}^{-1/2}G_m^t
D_{pqw}^{1/2}$ and note that the projection matrix $R_{m,p}=
\frac{1}{n} G_m
A_{m,p}^{-1}G_m^tD_{pqw}=\frac{1}{n}\tilde{E}^t_{m,p}\tilde{E}_{m,p}$.
Using the singular value decomposition and the definition of
$\tilde{E}_{m,p}$, we have
$$\|\tilde{E}^t_{m,p}\tilde{E}_{m,p}-\esp{}{\tilde{E}^t_{m,p}\tilde{E}_{m,p}}\|=\|\tilde{E}_{m,p}\tilde{E}^t_{m,p}-\esp{}{\tilde{E}_{m,p}\tilde{E}^t_{m,p}}\|,$$
since the singular values are the same. Thus, it is enough for our
purposes to bound
$\|\tilde{E}_{m,p}\tilde{E}^t_{m,p}-\esp{}{\tilde{E}_{m,p}\tilde{E}^t_{m,p}}\|$
in probability.

Next, we bound the rows of matrix $\tilde{E}^T_{m,p}$, $\tilde{e}^t_{m,p}(l)$. As before and using the bound in (\ref{cotainversa}) we have
$$\|\tilde{e}^t_{m,p}(l)\|\le c_m \sqrt{\frac{mQ}{p_{min}}} [1+(\sqrt{1+cn^{-\alpha-1}}\widetilde{\beta}_{m,k} +cn^{-\alpha-1})]^{1/2}.$$

On the other hand, because $R_{m,p}$ is a projection matrix,
$\|R_{m,p}\|=1$ and we have
\begin{eqnarray*}&&
1=\esp{}{\|R_{m,p}\|}\ge \sup_{\|u\|=1} \esp{}{\|R_{m,p}u\|}\\
&&\ge \sup_{\|u\|=1} \|\esp{}{R_{m,p}}u\|=\|\esp{}{R_{m,p}}\|
\end{eqnarray*}
so that $\Lambda_{\tilde{E}^t_{m,p}}\le 1$.
\begin{sloppypar}
 Then Lemma \ref{lemarauhut}
yields the stated result by the choice of the penalization
$\widetilde{B}_1(m,p,\delta)$ and  taking a union bound over $p \in
\mathcal{P}$.
\end{sloppypar}
\end{proof}

\begin{proof} of Lemma \ref{lemarauhut}

The proof  follows closely the ideas of the proof of Theorem 7.3 p. 62 in \cite{rauhut}.
Let $A\in \mathbb{R}^{n\times m}$ be a random matrix, with rows $a(l)\in
\mathbb{R}^{ m}$, l=1,\ldots n, satisfying
\begin{equation}\label{cotarauhut1}\|a(l)\|\le
K\sqrt{m}\end{equation}
 for some constant $K \geq 1$. Recall $A^tA=\sum_{l=1}^n a(l) a(l)^t$.

For the first part of the lemma we must bound
$E_r=\esp{}{\|\frac{1}{n
\Lambda_{A}}(A^tA-\esp{}{A^tA})\|^r}$, where  $\Lambda_{A}:=\frac{1}{n}\|\esp{}{A^tA}\|$.
Using the symmetrization Lemma  \citep[see][]{rauhut},  we have for all $2\leq r <
\infty$,
$$E_r \leq (\frac{2}{n\Lambda_{A}})^r \esp{}{\|\epsilon_l a(l) a(l)^t\|^r}.$$
where $\epsilon=(\epsilon_1,\ldots,\epsilon_n)$ is a Rademacher sequence
independent of $a(1), \ldots, a(n)$.

Thus, the following inequality holds
$$E_r\leq \left(\frac{2}{n\Lambda_{A}}\right)^r 2^{3/4} m r^{r/2} e^{-r/2} \esp{}{\|A\|^r\underset{l=1,\ldots,n}{\rm{max}} \|a(l)\|^r}, $$
where we have used    Rauhut's Lemma 6.18, p. 46, \cite{rauhut}  \cite[which is a version of Rudelson's  Lemma in][]{rudelson} and the Cauchy-Schwarz inequality  to obtain the stated result.

Furthermore, using the bound (\ref{cotarauhut1}) and applying the triangle inequality  yields
 \begin{eqnarray}
&&\nonumber \esp{}{\|A\|^r\underset{l=1,\ldots,n}{\rm{max}} \|a(l)\|^r}\\
&\nonumber \leq&\sqrt{\esp{}{\|A^tA\|^r}\esp{}{\underset{l=1,\ldots,n}{\rm{max}} \|a(l)\|^{2r}}}\\
&\label{cotarauhut2} \leq&  K^{r}m^{r/2} (n\Lambda_{A})^{r/2}((\esp{}{\|\frac{1}{n \Lambda_{A}}(A^tA-\esp{}{A^tA})\|^r})^{1/r} + 1)^{r/2},
\end{eqnarray}
where we have used  $\|\esp{}{1/nA^tA}/\Lambda_A\| = 1$.
Now, recall $$\Sigma_{r,m,n}=\left(\frac{2K}{\sqrt{n\Lambda_{A}}}\sqrt{\frac{m}{n}}\right)^r 2^{3/4} m r^{r/2} e^{-r/2}.$$
Then, using  the inequality (\ref{cotarauhut2})
$$E_r\leq \Sigma_{r,m,n} (E_r^{1/r} + 1)^{r/2}.$$
Whence, squaring the last inequality and completing  squares yields
$$\left(E_r^{1/r}-\Sigma_{r,m,n}/2\right)^2 \leq \Sigma_{r,m,n}^2 + \Sigma_{r,m,n}^4/4 .$$
 In the following we assume that $\Sigma_{r,m,n} \leq 1/2$. Thus,
 $$E_r^{1/r} \leq \sqrt{17}/4\Sigma_{r,m,n} + 1/4 \Sigma_{r,m,n} = \tau \Sigma_{r,m,n}.$$
where $\tau=(\sqrt{17}+1)/4$.

For the second part of the lemma we want to bound in probability
$\|\frac{1}{n \Lambda_{A}}(A^tA-\esp{}{A^tA})\|$.
The proof then follows directly from the first part of the lemma using the Markov inequality.

\end{proof}

\begin{proof} Proof of Lemma \ref{lema04}

For a given positive function $f$, recall
$\|u\|_{f,n}^2=1/n\sum_{i=1}^nf(t_i)u^{2}(t_i)$. Let $u\in S_m$ and
consider a linear application $A:R^n\to R^m$. Define
$\eta(A,z):=\sqrt{z^tA^tAz}$ for $z\in R^n $ and
$\eta_f(A,z):=\sup_{\|u\|_{f,n}=1}\sum_{i=1}^n f(t_i)z_i (A^tu)_i$.
Then, $\eta_f(A,z)=\eta(AD_f^{1/2},z)$, where $D_f$ is the diagonal
matrix with entries $f_i$.

On the order hand, note that
$\|R_{m,p}\varepsilon\|_{n,q}=\eta(R_{m,p}D_q^{1/2},\varepsilon)$.
The proof then follows directly from the next lemma, whose proof is
contained in \cite{Fermin--Ludena}.

\begin{lema} \label{lema1} Let
${\mathbf \varepsilon}=(\varepsilon_1,\ldots,\varepsilon_n)^t$ be a
vector of i.i.d. random variables  satisfying the moment condition
[MC].  Let $A$ be a given $m \times n$ matrix. Define
$\eta(A)=\eta(A,\varepsilon)=\sqrt{{\bf \varepsilon}^t A^t A {\bf \varepsilon}}$. Then,
for $r>1, u>0$ and $\theta>0$ there exists a positive constant $d$
that depends on $r$ such that the following inequality holds
\begin{eqnarray*}
&&P(\eta^2(A)\ge \sigma^2[Tr(A^tA)+ \rho(A^tA)]r(1+\theta)+\sigma^2 u)\\
\nonumber &&\le \exp\{-\sqrt{d (1/\rho(A^tA) u+ r \theta
[Tr(A^tA)/\rho(A^t A)+1] )}\}. \end{eqnarray*}
\end{lema}

To apply Lemma \ref{lema1}, we have to study the terms of the trace and the
spectral radius of the matrix
$\Gamma=(R_{m,p}D_q^{1/2})^tR_{m,p}D_q^{1/2}$. But, as $R_{m,p}$ is
a projection operator then $Tr(\Gamma)\le Q m$ and the spectral
radius $\rho(\Gamma)\le Q$.

Thus, we have, from the definition of ${\rm
\widetilde{B}}_2$ that
\begin{eqnarray*} && P\bigg(\sup_{\mathcal{P}}\left\{\|R_{m,p}{\bf
\varepsilon}\|_{n,q}^2-{\rm
\widetilde{B}}_2(m,p,\delta)\right\}
>0 \bigg)\\
 & \leq & \delta /2 \times \sum_k
\exp\{-\sqrt{d r \theta_k [m+1]}\}
\end{eqnarray*}
which yields the desired result.
\end{proof}

\begin{proof} of Theorem \ref{teoconsistencia1}

For any $p \in \mathcal{P}$,
\begin{eqnarray*}
&& \| x_m-\hat{x}_{m,\hat{p}}\|_{n,q}^2 \\
&&\leq \| x_m-\hat{x}_{m,p}\|_{n,q}^2  +  \left\lbrace \| x_m-\hat{x}_{m,\hat{p}}\|_{n,q}^2 -\widetilde{B}(m,\hat{p},\delta,\gamma,n)\right\rbrace \\
&& \quad + \left\lbrace \widetilde{B}(m,p,\delta,\gamma,n) - \|
x_m-\hat{x}_{m,p}\|_{n,q}^2\right\rbrace
\end{eqnarray*}
where $\widetilde{B}(m,p,\delta,\gamma,n)=(1+\gamma)
\widetilde{B}_1(m,p,\delta)+(1+1/\gamma)\widetilde{B}_2(m,p,\delta)$ with $\widetilde{B}_1$ and
$\widetilde{B}_2$ defined in (\ref{pen1-1}) and  (\ref{pen2}).

On the other hand,  recall that (see equation \eqref{eq:descerror}),
$$x_m-\hat{x}_{m,p} = \esp{}{R_{m,p}}[x_{0}-x_{m}]+(R_{m,p}-\esp{}{R_{m,p}})[x_{0}-x_{m}]+ R_{m,p}\varepsilon.$$ Since for any $0<\gamma<1$,  $2ab \leq \gamma a^2 + 1/\gamma b^2$ holds
for all $a,b \in {\rm I\!R}$, following standard arguments we have
\begin{eqnarray*}
&&\| x_m-\hat{x}_{m,p}\|_{n,q}^2 \\
&& \leq 2
\|\esp{}{R_{m,p}}[x_{0}-x_{m}\|^2 + 2(1+\gamma) \|(R_{m,p}-\esp{}{R_{m,p}})[x_{0}-x_{m}]\|^2\\
&& \quad + 2(1+1/\gamma) \|R_{m,p}\varepsilon\|_{n,q}^2.
\end{eqnarray*}

Thus,
\begin{eqnarray*}
&& \| x_m-\hat{x}_{m,\hat{p}}\|_{n,q}^2 \\
&& \leq  6 \|\esp{}{R_{m,p}}(x_m-x_0)\|^2_{n,q} \\
&& \quad + 6(1+\gamma)\widetilde{B}_1(m,\hat{p},\delta)+6(1+1/\gamma)\widetilde{B}_2(m,\hat{p},\delta) \\
&& \quad + 6 (1+\gamma) ( \sup_{\mathcal{P}}\left\lbrace \|R_{m,p}(x_m-x_0)-\esp{}{R_{m,p}}(x_m-x_0)\|^2_{n,q}\right.\\
&& \quad - \widetilde{B}_1(m,p,\delta)\left. \right\rbrace) \\
&& \quad + 6(1+\gamma^{-1}) ( \sup_{\mathcal{P}}\left\lbrace
\|R_{m,p}\varepsilon\|^2_{n,q}- \widetilde{B}_2(m,p,\delta)\right\rbrace)
\end{eqnarray*}
Finally, as follows from Lemma \ref{lema03} and \ref{lema04}, with
probability greater than $1-\delta$, we have the stated result.
\end{proof}

\begin{proof} of Lemma \ref{lema3--1}

Recall, from Lemma \ref{lema03},  $A_{m,p}=1/nG_m^tD_{w,q,p}G_m$ and set $A_{m}=\esp{}{A_{m,p}}=1/nG_m^tD_{q}G_m$. Then
$R_{m,p}=1/nG_m A_{m,p}^{-1}G_m^tD_{w,q,p}$ and
$$R_{m}=\esp{}{1/nG_m A_{m}^{-1}G_m^tD_{w,q,p}}=1/nG_m A_{m}^{-1}G_m^tD_{q}.$$
 Remark that under condition [AS], $\|A_{m}-I\|\le c n^{-1-\alpha}$.

Set $Q_{m,p}=A_{m,p}-I$, so that $\|\esp{}{Q_{m,p}}\| \le c n^{-1-\alpha}$.
Set $K_A:=c_m\sqrt{Q/p_{\min} }$.   By Lemma \ref{lemarauhut} we have, for any $r\ge 2$,
\begin{eqnarray*} && E_{m,p,r}^{1/r}:=[\esp{}{\|{A_{m,p}}-\esp{}{A_{m,p}}\|}^r]^{1/r}\\
&&\le \tau \left(\frac{2K_A}{\sqrt{n\varLambda}}\sqrt{\frac{m}{n}}\right)^r \varLambda 2^{3/4} m r^{r/2} e^{-r/2}.
\end{eqnarray*}
where $\varLambda=\|\esp{}{A_{m,p}}\|=O(1)$.
\begin{sloppypar}
Whence, for big enough $r$, since $m=o(n)$ and $p_{\min}^{-1}$ is fixed, we have   $E_{m,p,r}^{1/r} =O(n^{-1-\alpha})$ and thus  by H\"older's inequality that
$\esp{}{\|Q_{m,p}\|}\le O(  n^{-1-\alpha})$. The latter yields that, in particular, $\esp{}{\|Q_{m,p}\|} <1$.
Set $C_{m,p}:=A_{m,p}^{-1}-I$. Using the classical Neumann expansion, under condition [AS], by the Monotone Converge Theorem we may finally bound $\esp{}{\|C_{m,p} \|}\le c n^{-1-\alpha} $ for a certain positive constant $c$. We also have $A_{m}^{-1}=I+C_{m}$ with $\|C_{m} \|\le n^{-1-\alpha}$.
\end{sloppypar}
On the other hand remark, from the definition of the spectral norm, that for any matrix $B$,  $\|B\|=\|B^t\|=\sqrt{\|BB^t\|}$,
so that for any given matrix $M$,
\begin{eqnarray*}
&&\|1/nG_m MG_m^tD_{w,q,p}[x_0-x_m]\|_{n,q}\le \|M\| \|1/nG_m G_m^t\| \|D_{w,q,p}\| \|[x_0-x_m]\|_{n,q}\\
&&\le \|M\| \|[x_0-x_m]\|_{n,q}/p_{\min},
\end{eqnarray*}
where the last bound follows from the definition of the diagonal matrix $D_{w,q,p}$ and the bounds on $\|1/nG_m G_m^t\|$ under condition [AS].

Then, since by definition $R_m[x_0-x_m]=0$ we have

\begin{eqnarray*}
&& \|\esp{}{R_{m,p}}[x_0-x_m]\|_{n,q}\\
&&=\|\esp{}{1/nG_m[I+C_{m,p}+A_{m}^{-1} -A_{m}^{-1}]G_m^tD_{w,q,p}}[x_0-x_m]\|_{n,q}\\
&& \le\|\esp{}{1/nG_m C_{m,p}G_m^tD_{w,q,p}}[x_0-x_m]\|_{n,q}\\
&&+\|1/nG_m[I-A_{m}^{-1}]G_m^t\esp{}{D_{w,q,p}}[x_0-x_m]\|_{n,q}\\
&&\le \esp{}{ \|  C_{m,p}\| \|1/nG_mG_m^t\| \|D_{w,q,p}\|} \|x_0-x_m\|_{n,q}\\
&&+\|I-A_{m}^{-1}\| \|1/nG_mG_m^t\| \|\esp{}{D_{w,q,p}}\| \|x_0-x_m\|_{n,q}\\
&& \le c[\frac{n^{-1-\alpha}}{p_{min}}+n^{-1-\alpha}]\|x_0-x_m\|_{n,q},
\end{eqnarray*}
where, for the line before last we have used $\|\esp{}{B}\| \le \esp{}{\|B\|}$ for any given matrix $B$, and the last line follows from the above discussion.

\end{proof}

\begin{proof} of Theorem \ref{teoconsistencia2}

The proof follows from Lemma \ref{lema2} below
\end{proof}

In order to state Lemma \ref{lema2} we introduce
for any given $p$  and $x\in S_m,x'\in S_{m'}$, the quantities $$\Delta_1(x,x',p):=[L_{n}(x,y,p)-L_{n}(x,p)]-[L_{n}(x',y,p)-L_{n}(x',p)]$$
and $$\Delta_2(x,x',p):=[L_n(x,p)-L(x)]-[L_n(x',p)-L(x')].$$

We then have.

 \begin{lema}\label{lema2}
   Let
${\mathbf \varepsilon}$ be a
vector of i.i.d. random variables  satisfying the moment condition
[MC]. Assume that the conditions [AB], [AS] and [AQ] are satisfied. Let $\mathcal{P}:=\{p_{k}, k\ge 1 \}$
be a numerable collection of   $[0,1]$ valued functions and set $p_{k,min}=\min_i p_{k}(t_{i})$. Assume  that $\min_{k}p_{k,min}>p_{\min}$
and $\theta_{m,k} \geq 0$, such that  the following Kraft inequality $\sum_{m,k}e^{-\sqrt{d r \theta_{m,k}(d_m+1) }} < 1$ holds.
Assume $pen_0$, $pen_1$ and $pen_2$ to be selected according to  (\ref{pen0}), (\ref{pen1new}) and (\ref{pen2new}) respectively. Let $x\in S_m$ and $x'\in S_{m'}$. Then
 \begin{eqnarray*}
 &&P(\sup_{m,m',k}( \Delta_1(x,x',p_k)\\
 &&-[\frac{1}{\gamma p_{\min}^2}(pen^2_2(m,p_k,\delta)+pen^2_2(m',p_k,\delta))\\
 &&+
 \gamma(\|x_0-x\|_{n,q}^2+\|x_0-x'\|_{n,q}^2)])>0)\\
 &\le& \delta/3
 \end{eqnarray*}
 and
 \begin{eqnarray*}
 &&P(\sup_{m,m',k}\{\Delta_2(\hat{x}_{m,p},x',p_k)\\
 &&-
 [2p_0(m,p_k,\delta)+(\frac{ {1}}{p_{\min}}+\frac{ {1}}{\gamma})pen_{1}(m,p_k,\delta)\\
 &&+2p_0(m',p_k,\delta)+
 \frac{ {1}}{p_{\min}}pen_{1}(m',p_k,\delta)\\
 &&+
 3\gamma\|x_0-x_m\|_{n,q}^2+(\frac{1}{p_{\min}^2}(\frac{1}{\gamma}+1)+\frac{1}{\gamma})pen_2(m',p_k,\delta)\\
 &&+2((c+1)\frac{n^{-(1+\alpha)}QC}{p_{min}})^2]\}>0)\\
 &\le& 2\delta/3
 \end{eqnarray*}
 \end{lema}

\begin{proof}

In order to simplify notation throughout the proof of the lemma we use $p$ instead of $p_k$.
The first part of Lemma \ref{lema2} is rather standard, the only
care being taking into account the random norm.  For any $x\in S_m$ and $x'\in S_{m'}$ we have

 \begin{eqnarray*}
&&|\Delta_1(x,x',p)|\\
&=&|\frac{2}{n}\sum_{i=1}^n q(t_i) \frac{w_i}{p(t_i)}\varepsilon_i (x-x')(t_i)|\\
&\le&|\frac{2}{n}\sum_{i=1}^n q(t_i) \frac{w_i}{p(t_i)}\varepsilon_i (x_0-x)(t_i)|\\
&&+|\frac{2}{n}\sum_{i=1}^n q(t_i) \frac{w_i}{p(t_i)}\varepsilon_i (x_0-x')(t_i)|\\
&\le&2\|(x_{0}-x)\|_{n,qw/p} \sup_{\|v\|_{n,qw/p}=1, v\in S_m }\frac{1}{n}\sum_{i=1}^n q(t_i) \frac{w_i}{p(t_i)}\varepsilon_iv_i\\
&&+2\|(x_{0}-x')\|_{n,qw/p}\sup_{\|v\|_{n,qw/p}=1, v\in S_{m'} }\frac{1}{n}\sum_{i=1}^n q(t_i) \frac{w_i}{p(t_i)}\varepsilon_iv_i\\
&\le& \gamma(\|x_0-x\|_{n,q}^2+\|x_0-x'\|_{n,q}^2)\\
&&+\frac{1}{\gamma p_{\min}^2}(\|R_{m,p}\varepsilon\|_{n,qw/p}^2+
\|R_{m',p}\varepsilon\|_{n,qw/p}^2),
\end{eqnarray*}
where we have used $\|R_{m,p}\varepsilon\|_{n,qw/p}=\sup_{\|v\|_{n,qw/p}=1}\frac{1}{n}\sum_{i=1}^n q_i\frac{w_i}{p_i} \varepsilon_iv_i$ and
the inequality $2ab\le \gamma a^2+1/\gamma b^2$
to obtain the stated result.

Hence,
\begin{eqnarray}
&&\nonumber P_w(\sup_{m,m',p}[\Delta_1\\
&&\nonumber-\frac{1}{p_{\min}^2 \gamma} (pen_2(m,p,\delta)+
pen_2(m',p,\delta))\\
&&\nonumber-\gamma(\|x_0-x\|_{n,q}^2+\|x_0-x'\|_{n,q}^2)]>0)\\
&\le&\nonumber 2\sum_{m,k}P_{w}(\|R_{m,p}\varepsilon\|_{n,qw/p}^2>pen_2(m,p_k,\delta))\\
&\le& \label{cota30}\delta/3
\end{eqnarray}
by Lemma \ref{lema1} and the choice of the penalization $pen_2$ in (\ref{pen2new}), recalling $R_{m,p}$ is a projection matrix.

The term $\Delta_2$ requires a little more work. To begin with, for
any $x \in S_m$, write
$\widetilde{L_n}(x,p):=\|x-x_{0}\|^2_{n,qw/p}$ and
$\widetilde{L}(x):= \|x-x_{0}\|^2_{n,q}$. Recall that, for a given
$m$,
$\hat{x}_{m,p}-x_0=R_{m,p}(x_0-x_m)+(x_m-x_0)+R_{m,p}\varepsilon$.
To deal with this term, we must consider all the terms in the square
of this expression. Thus,
\begin{eqnarray*}
&&\widetilde{L}_n(\hat{x}_{m,p},p)-\widetilde{L}(\hat{x}_{m,p},p)=\frac{1}{n}\sum_{i=1}^n q(t_i)(\frac{w_i}{p_i}-1)[R_{m,p}(x_0-x_m)]^2(t_i)\\
&&+\frac{2}{n}\sum_{i=1}^n q(t_i)(\frac{w_i}{p_i}-1)R_{m,p}(x_0-x_m)(t_i)(x_m-x_0)(t_i)\\
&&+\frac{1}{n}\sum_{i=1}^n q(t_i)(\frac{w_i}{p(t_i)}-1)(x_m-x_0)^2(t_i)\\
&&+\frac{1}{n}\sum_{i=1}^n q(t_i)(\frac{w_i}{p(t_i)}-1)[R_{m,p}\varepsilon]_i^2\\
&&+\frac{2}{n}\sum_{i=1}^n q(t_i)(\frac{w_i}{p(t_i)}-1)[R_{m,p}\varepsilon]_i[R_{m,p}(x_0-x_m)](t_i)\\
&&+\frac{2}{n}\sum_{i=1}^n q(t_i)(\frac{w_i}{p(t_i)}-1)[R_{m,p}\varepsilon]_i (x_0-x_m)(t_i)\\
&&=\mathbf{I}_a +\mathbf{I}_b +\mathbf{I}_c+\mathbf{I}_d+\mathbf{I}_e+\mathbf{I}_f
\end{eqnarray*}
Start with $\mathbf{I}_a$.

Write
\begin{eqnarray*}
&&\|R_{m,p}(x_0-x_m)(t_i)\|_{n,qw/p}^2-\|R_{m,p}(x_0-x_m)(t_i)\|_{n,q}^2\\
&&=\frac{1}{n}\sum_{i=1}^n q(t_i)(\frac{w_i}{p(t_i)}-1)[R_{m,p}(x_0-x_m)]^2(t_i).
\end{eqnarray*}
 Note that
\begin{eqnarray*}
 &&\|[R_{m,p}-\esp{}{R_{m,p}}](x_0-x_m)(t_i)\|_{n,qw/p}^2-\|[R_{m,p}-\esp{}{R_{m,p}}](x_0-x_m)(t_i)\|_{n,q}^2\\
&&\leq \frac{1} {p_{min}} \|[R_{m,p}-\esp{}{R_{m,p}}]\| \left(\|(x_0-x_m)(t_i)\|_{n,q}^2\right),
\end{eqnarray*}
Whence from  the choice of $pen_1$, using  Lemma \ref{lema03} and   summing over $m$ we obtain

\begin{eqnarray} &&\nonumber P(\sup_{m}\sup_p[\|[R_{m,p}-\esp{}{R_{m,p}}](x_0-x_m)(t_i)\|_{n,qw/p}^2\\
\nonumber&&-\|[R_{m,p}-\esp{}{R_{m,p}}](x_0-x_m)(t_i)\|_{n,q}^2-pen_{1}(m,p,\delta)] )\\
&& \label{cota32} <\delta/6.
\end{eqnarray}
We then use Lemma \ref{lema3--1} to bound $\|(\esp{}{R_{m,p}}-R_m)[x_0-x_m]\|_{n,q}= (c+1)n^{-1-\alpha}e_m/p_{min}$ and achieve the bound of the term $\mathbf{I}_a$.

For the term $\mathbf{I}_b$ we start by remarking that
$$\sum_{i=1}^n q(t_i)\frac{w_i}{p(t_i)}R_{m,p}(x_0-x_m)(t_i)(x_m-x_0)(t_i)=0$$ by orthogonality. The term
$$\sum_{i=1}^n q(t_i)R_{m,p}(x_0-x_m)(t_i)(x_m-x_0)(t_i)$$ is then bounded
$$[\sum_{i=1}^n q(t_i)R_{m,p}(x_0-x_m)(t_i)(x_m-x_0)(t_i)]^2\le \gamma\|x_0-x_m\|_{n,q}^2+\frac{1}{\gamma}\|R_{m,p}[x_0-x_m]\|_{n,q}^2$$ and the proof follows as for $\mathbf{I}_a$.

 For $\mathbf{I}_c$, the proof follows from Lemma \ref{lemanorm} below,
  \begin{equation}\label{cota33}
  P(\sup_{m}\sup_p\{\frac{1}{n}\sum_i q(t_i)(\frac{w_i(p)}{p(t_i)}-1) (x_m-x_0)^2(t_i)-p_0(m,p,\delta)\}>0)<\delta/6.
  \end{equation}

  For $\mathbf{I}_d$, Lemma \ref{lema1} implies that
  \begin{equation}\label{cota34}
  P(\sup_m\sup_p\{\frac{1}{n}\sum_i q(t_i)(\frac{w_i(p)}{p(t_i)}-1)[R_{m,p}\varepsilon]_i^2  -\frac{1}{p_{\min}^2}p_2(m,p,\delta)\}>0)<\delta/6.
  \end{equation}
  The term $\mathbf{I}_e$ follows exactly as for $\Delta_1$. Finally, as for $\mathbf{I}_b$, by orthogonality we only have to bound the term
  $$[\sum_{i=1}^n q(t_i)[R_{m,p}\varepsilon]_i(x_m-x_0)(t_i)]^2$$ whose proof then follows exactly as for $\mathbf{I}_d$.

The proof then follows  by gathering the bounds in (\ref{cota30}),  (\ref{cota32}),
(\ref{cota33}) and (\ref{cota34}).
\end{proof}

 \begin{lema}\label{lemanorm} Assume that there exists a positive constant $C$ such that  $
\sup_{m} e_m\le C$ with $e_m=\|x_0-x_m\|_{\infty}$. Assume that the condition  [AQ] is satisfied. Let $\mathcal{P}:=\{p_{k}, k\ge 1 \}$
be a numerable collection of   $[0,1]$ valued functions and set $p_{k,min}=\min_i p_{k}(t_{i})$. Assume  that $\min_{k}p_{k,min}>p_{\min}$. Assume $pen_0$ to be selected according to  (\ref{pen0}).
Then,
 $$ P(\sup_m\sup_p \{\|x_0-x_m\|_{n,qw/p}^2-\|x_0-x_m\|_{n,q}^2 - pen_0(m,p,\delta)\}>0)<\delta/6$$
\end{lema}
\begin{proof}

Note that

\begin{eqnarray*}
&&\|x_0-x_m\|_{n,qw/p}^2-\|x_0-x_m\|_{n,q}^2 \\
&&=\frac{1}{n}\sum_{i=1}^n
q(t_i)(\frac{w_i}{p_i}-1)(x_0-x_m)^2(t_i).
\end{eqnarray*}
Let $p^*$ attain the supremum of this expression, so that
\begin{eqnarray*}
&&\|x_0-x_m\|_{n,qw/p^*}^2-\|x_0-x_m\|_{n,q}^2 \\
&&=\sup_p\left\{\|x_0-x_m\|_{n,qw/p}^2-\|x_0-x_m\|_{n,q}^2\right\}.
\end{eqnarray*}
Since $x-x_0$ is not random and is uniformly bounded
\begin{eqnarray*}
&&\esp{}{\sup_p\frac{1}{n}\sum_i q(t_i)(\frac{w_i(p)}{p_i}-1) (x_0-x_m)^2(t_i)}\\
&&=\frac{1}{n}\sum_i q(t_i)(x_0-x_m)^2(t_i)\esp{}{(\frac{w_i(p*)}{p^*_i}-1) }=0,
\end{eqnarray*}
Whence from  the choice of $pen_0$ in (\ref{pen0}), using  the bounded differences inequality  \citep{essaimsurvey},   we have
$$ P(\sup_m\sup_p\{ \|x_0-x_m\|_{n,qw/p}^2-\|x_0-x_m\|_{n,q}^2 - pen_0(m,p,\delta)\}>0)\leq\sum_m \frac{\delta}{6m(m+1)},$$  which yields the desired result.
\end{proof}

\begin{proof} of  Theorem
  \ref{teoiterativo}.

   The proof   is based on the following preliminary Lemma. 
  \begin{lema}\label{lemaitera} For any $\delta>0$, with probability at least $1-\delta$, for all $j\le n-n_0$ and all $x,x'\in S_{j-1}$
  $$|L_j(x)-L_j(x')-[L(x)-L(x')]|\le \Delta_j.$$
  \end{lema}
  Set $\delta>0$ so from  lemma \ref{lemaitera} $$|L_j(x)-L_j(x')-[L(x)-L(x')]|\le \Delta_j$$ holds for
  all $0\le j\le n-n_0$, $x,x'\in S_{(j-1)\vee 0}$ with probability at least $1-\delta$. Hence, for any $x,x'\in S_{j-1}$, since $S_{j-1}\subset S_{j-2}$, from the definition of $S_{j-2}$
  \begin{eqnarray*}
 && L(x)-L(x')\le L_{j-1}(x)-L_{j-1}(x')+ \Delta_{j-1}\\
 &&\le L_{j-1}(\hat{x}_{j-1})+\Delta_{j-1}-L_{j-1}(\hat{x}_{j-1})+\Delta_{j-1}=2\Delta_{j-1}.
  \end{eqnarray*}

  On the other hand, over the chosen event with probability greater than $1-\delta$,  by the choice of $\Delta_0$  and the results in Section \ref{subsection1.3},
  $x^*\in S_0$ from the definition of $S_0$. We shall now prove by induction that over the stated  event $x^*\in S_{j}$ for $1\le j\le n-n_0$.
  Assume $x^*\in S_{j-2} $. By lemma \ref{lemaitera},
  $$L_{j-1}(x^*)-L_{j-1}(\hat{x}_{j-1})\le L(x^*)-L(\hat{x}_{j-1})+\Delta_{j-1}\le \Delta_{j-1},$$
  so that $x^*\in S_{j-1}$, which ends the proof by induction.
 
   Whence,    for all $1\le j\le n-n_0$, $L(\hat{x}_{j })\le L(x^*)+2\Delta_{j-1}$, which ends the proof of the Theorem.

  \end{proof}

  \begin{proof} of lemma \ref{lemaitera}:

  Recall $n_j=n_0+j$ with $j=0,\ldots,n-n_0$. For fixed $j$ and any $x,x'\in S_{j-1}$ we have
  \begin{eqnarray*}
  && L_{j}(x)-L_j(x')-[L(x)-L(x')]\\
  &&= \frac{1}{n_j}\sum_{i=1}^{n_j} q(t_{j_i}) (\frac{w_{j_i}}{p(t_{j_i})}-1)(x-x')(t_{j_i})(x+x'-2x_0)(t_{j_i})\\
  &&+\frac{2}{n_j}\sum_{i=1}^{n_j} q(t_{j_i}) \frac{w_{j_i}}{p(t_{j_i})}\varepsilon_i (x(t_{j_i})-x'(t_{j_i}))\\
  &&=\mathbf{I}_j+\mathbf{II}_j.
  \end{eqnarray*}
  Set $z_i=q(t_{j_i})w_i (x-x')(t_{j_i})(x+x'-2x_0)(t_{j_i})$, so that $\|z_i\|_\infty\le 2QB_j(2B_j\wedge 1)$ and $\mathbf{I}_j$ satisfies the bounded difference inequality with
  $c^2= 4Q^2B_j^2(2B_j\wedge 1)^2$.  By equation (\ref{cota52}) in Section \ref{subsection3} we have
  \begin{eqnarray*}
  && P(\mathbf{I}_j>\sqrt{\log(4(n_j)(n_j +1)/\delta)\frac{16B_j^2(2B_j\wedge
        1)^2Q^2}{n_j}}+4\sqrt{4\frac{(d_j+1)\log  n}{n_j}}) \\
    &&\le P(\mathbf{I}_j>\sqrt{\log(4(n_j)(n_j +1)/\delta)\frac{16B_j^2(2B_j\wedge
       1)^2Q^2}{n_j}}+4\esp{}{ \Omega_j})\\
  &&\le \delta/2((n_j)(n_j +1)).
  \end{eqnarray*}
\begin{sloppypar}
  Next we deal with $\mathbf{II}_j$. Set $u(t)=\frac{w_i (x-x')(t)}{p(t_i)}\frac{1}{\sqrt{Q}}$, so that $\frac{1}{n_j}\sum_{i}q(t_{j_i})u^2(t_{j_i})\le 1$. Then,
  \end{sloppypar}
  \begin{eqnarray*}
  &&|\mathbf{II}_j|\le \sup_{u\in S_{j-1}, \|u\|_{n_j,q}}\frac{\sqrt{Q}}{n_j}\sum_i q(t_{j_i}) u(t_{j_i}) \varepsilon_i \\
  &&=\sqrt{Q}\|\Pi_{S_{j-1}}\varepsilon\|_{n_j,q},
  \end{eqnarray*}
where $\Pi_{S_{j}}$ stands for the projection over $S_j$. Whence by Lemma \ref{lema1},
 \begin{eqnarray*}&&P\left(\mathbf{II}_{j}>\sqrt{\sigma^2Q[(\frac{  2(d_j+
 1)}{n_j})+\frac{\log^2(4(n_j)(n_j +1)/\delta)}{n_j}]}\right)\\
 &&\, \, \, \le \delta/2((n_j)(n_j +1)).\end{eqnarray*}
   Summing over $j$ ends the proof.

\end{proof}

\begin{proof} of lemma \ref{simplesize}:
	\begin{eqnarray*}
		(p_j(t))^2&\le & \sup_{x,x'\in S_j} \|x-x'\|_\infty^2\le 4 \sup_{x\in S_j} \|x-x^*\|_\infty^2\\
		&&\le 4\overline{r}^2 d_j \sup_{x\in S_j}L(x-x^*)\\
		&&\le 4\overline{r}^2 d_j \sup_{x\in S_j}[L(x)+L(x^*)]\\
		&&\le 4\overline{r}^2 d_j(2L(x^*)+2\Delta_j).
	\end{eqnarray*}
	\begin{sloppypar}
		The third inequality follows from  Lemma \ref{lemamassartbirge} and the fifth from the bound $\sup_{x\in S_j}L(x)\le L(x^*)+2\Delta_j$ as follows from Theorem \ref{teoiterativo}. The proof is achieved by calculating the square root of each side of the last series of inequalities and finally using that $\sqrt{a+b}\le \sqrt{a}+\sqrt{b}$.
	\end{sloppypar}
	
\end{proof}



\end{document}